\documentclass[10pt,a4paper]{article}
\usepackage[utf8x]{inputenc}
\usepackage{amsmath, amsthm}
\usepackage{amsfonts}
\usepackage{amssymb}
\usepackage{parskip}

\newtheorem{theorem}{Theorem}

\newtheorem{conj}{Conjecture}
\newtheorem{coro}{Corollary}

\theoremstyle{definition}
\newtheorem{exmp}{Example}

\begin{document}
\title{On the Hilbert series of ideals generated by generic forms}
\author{Lisa Nicklasson}

\maketitle

\begin{abstract}
\noindent
There is a longstanding conjecture by Fröberg about the Hilbert series of the ring $R/I$, where $R$ is a polynomial ring, and $I$ an ideal generated by generic forms. We prove this conjecture true in the case when $I$ is generated by a large number of forms, all of the same degree. We also conjecture that an ideal generated by $m$'th powers of forms of degree $d$ gives the same Hilbert series as an ideal generated by generic forms of degree $md$. We verify this in several cases. This also gives a proof of the first conjecture in some new cases. 
\end{abstract}

\section{Introduction}
Let $R=\mathbb{C}[x_1,\ldots,x_n]$, the ring of polynomials in $n$ variables with complex coefficients (in fact, a polynomial ring over any field with characteristic 0 will do). Let $I$ be an ideal of $R$, generated by homogeneuos polynomials. We consider 
\[ R/ I = \bigoplus_{i \geq 0} R_i\]
as a graded ring in the usual sense. Now, let $\alpha_i=\dim_\mathbb{C}R_i$. We define the \emph{Hilbert series} of $R/I$ as the formal power series
\[H_{R/I}(t)= \sum_{i \geq 0}\alpha_it^i. \]
In this paper we shall study the Hilbert series of $R/I$, for some different choices of $I$. 

For a power series $F(t)=\sum_{i\geq 0} a_it^i$ define 
\[\lceil F(t) \rceil=\sum_{i\geq 0} b_it^i ~\textrm{where}~ b_i=a_i ~ \textrm{if} ~ a_j>0 ~ \textrm{for all} ~j\leq i, ~\textrm{and}~ b_i=0~ \textrm{otherwise.} \]
That is, we include the terms of $F(t)$ as long as the coefficients are positive. \\
Let $G(t)=\sum_{i \ge 0}c_it^i$. We write $F(t) \succeq G(t)$ if $F(t)$ is greater than or equal to $G(t)$ in the lexicographical sense. That is, $F(t) \succeq G(t)$ if $F(t)=G(t)$, or there is some $i$ such that $a_i > c_i$ and $a_j = c_j$ for all $j<i$. 

Let $(m_i)$ denote all monic monomials of degree $d$ in $R$. In the continuation we shall just use the word \emph{monomial} for monic monomials. There are $ n+d-1 \choose d$ monomials of degree $d$. A form $\sum \beta_i m_i$ of degree $d$ can be identified with its $ n+d-1 \choose d$ coefficients $(\beta_i)$. The form is called \emph{generic} if all the $ n+d-1 \choose d$ coefficients are algebraically independent. A family of forms is called generic if all the coefficients from all the forms are algebraically independent. 

Let us now consider ideals generated by $k$ forms $g_1,g_2, \ldots, g_k$ and $\deg g_i = d_i$. It is proved in \cite{Frob._Lofw.} that all such ideals, where the forms are generic, give rise to the same Hilbert series of $R/I$. It is also shown in \cite{Frob._ineq} that this series is the smallest possible Hilbert series for $R/I$. It is not completely known what this series is, but there is the following conjecture.\\

\begin{conj}[Fröberg 1985]
 If $I=(g_1,g_2, \ldots, g_k)$ is an ideal of $R$, generated by generic forms, and $\deg g_i=d_i$, then 
 \[
  H_{R/I}(t)=\left\lceil \frac{\prod_{i=1}^k(1-t^{d_i})}{(1-t)^n} \right\rceil.
 \]
\end{conj}
It is proved in \cite{Frob._ineq} that 
\[
 H_{R/I}(t) \succeq \left\lceil \frac{\prod_{i=1}^k(1-t^{d_i})}{(1-t)^n} \right\rceil
\]
which means that is is enough to find one ideal $I$ which gives the above Hilbert series, to prove the conjecture true for that particular choise of $k, d_1, \ldots, d_k$ and $n$. The conjecture has been proved true by Fröberg for $n=2$ (see \cite{Frob._ineq}), Anick for $n=3$ (see \cite{Anick}), and Stanley for $k=n+1$ (see \cite[Example 2, p.127]{Frob._ineq} ). It is also true for $k \leq n$, since $I$ then is a complete intersection. In this paper we shall study the case when all the forms $g_i$ have the same degree $d$. We shall then prove that $I$ gives the conjectured Hilbert series for large $k$. We shall also study the case when the forms $g_i$ are powers of generic forms. 

\section{Ideals generated by a large number of forms}
From now on we shall consider ideals genereted by $k$ forms $g_1, \ldots, g_k$, all of degree $d$. As noted earlier, there are ${{n+d-1} \choose d}$ monomials in $R$ of degree $d$. If $k >  {{n+d-1} \choose d}$ the forms $g_1, \ldots, g_k$ have to be linearly dependent, and they are not all needed as generators. In this way, we see that the number of generators can always be assumed to be less than or equal to ${{n+d-1} \choose d}$. Next we shall see that if ${{n+d-1} \choose d}-n < k \leq {{n+d-1} \choose d}$ the forms $g_1, \ldots, g_k$ can be chosen to be monomials. To do this, we shall use the following theorem. \\

\begin{theorem}
Let $\mathfrak{m}=(x_1, \ldots x_n)$, and let $I$ be an ideal such that ${\mathfrak{m}^{d+1} \subset I \subseteq \mathfrak{m}^d}$. Assume that $r= \dim_\mathbb{C}(I/\mathfrak{m}^{d+1})$ satisfies $r \geq {{n+d} \choose {d+1}} \big / n$. Then
\[
 H_{R/I}(t)= \left\lceil \frac{(1-t^d)^r}{(1-t)^n} \right\rceil.
\]
\end{theorem}
\begin{proof}
 We easily see that 
 \[
  H_{R/I}(t)=\sum_{i=0}^{d-1} {{n+i-1} \choose i}t^i + \left(\!{{n+d-1} \choose d}-r\right)\!t^d.
 \]
Note that
\[
 \left\lceil \frac{(1-t^d)^r}{(1-t)^n} \right\rceil = \left\lceil \left( \sum_{i=0}^\infty {{n+i-1} \choose i}t^i \right)\!( 1-rt^d + \ldots ) \right\rceil.
\]
The coefficient of $t^d$ here is ${{n+d-1} \choose d}-r$. The coefficient of $t^{d+1}$ is ${{n+d} \choose {d+1}} -rn$, which is nonpositive by the assumption on $r$. Hence
\[
 \left\lceil \frac{(1-t^d)^r}{(1-t)^n} \right\rceil =\sum_{i=0}^{d-1} {{n+i-1} \choose i}t^i + \left(\!{{n+d-1} \choose d}-r\right)\!t^d = H_{R/I}(t).
\] 
\end{proof}
Note that the number of generators of $I$ in the theorem is at least $r$. If we find such an ideal, generated by exactly $r$ forms of degree $d$ we have proved the Fröberg conjecture true for this number of generators. Now, assume $d\geq2$ and $n \geq 4$. Let $I$ be the ideal generated by all monomials of degree $d$, except 
\[
 x_1x_2^{d-1},~x_1x_3^{d-1}, \ldots ,~x_1x_l^{d-1},
\]
where $1 \leq l \leq n$. The case $l=1$ is interpreted as $I=\mathfrak{m}^d$. We shall see that $\mathfrak{m}^{d+1} \subset I$. What needs to be proved is that any monomial of degree $d+1$ which comes from multiplying one of the above by $x_i$ is contained in $I$. If $i \neq 1$ and $2 \leq j \leq n$ we have $x_ix_1x_j^{d-1} = x_1x_ix_j^{d-1} \in I.$ If $i=1$ we have $x_1^2x_j^{d-1} = x_jx_1^2x_j^{d-2} \in I$. \\
The ideal $I$ is generated by $k={{n+d-1} \choose d} -l+1$ monomials. We also need to prove that $k \geq {{n+d} \choose {d+1}} \big / n$. That is, we need to prove
\[
 {{n+d} \choose {d+1}} -n{{n+d-1} \choose d} +n(l-1) \leq 0
\]
We shall prove 
\[
 \left( \frac{n+d}{d+1}-n \right) \!\! {{n+d-1} \choose d} \leq -n^2
\]
by induction over $d$, and then the above follows. First we check for $d=2$. This gives
\begin{align*}
 \left(\frac{n+2}{3}-n \right) \!\! {{n+1} \choose 2} &= \frac{2(1-n)}{3} {{n+1} \choose 2} \\[1ex]
 &= \frac{2(1-n)}{3} \frac{n(n+1)}{2} \\[1ex]
 &= -\frac{(n-1)n(n+1)}{3}\leq -n^2
\end{align*}
which is true since $n \geq 4$. Now assume 
\[
 \left( \frac{n+d}{d+1}-n \right) \!\! {{n+d-1} \choose d} \leq -n^2.
\]
We want to show
\[
 \left( \frac{n+d+1}{d+2}-n \right) \!\! {{n+d} \choose {d+1}} \leq -n^2.
\]
It follows that
\begin{align*}
 \left( \frac{n+d+1}{d+2}-n \right) \!\! {{n+d} \choose {d+1}} =& \frac{(d+1)(1-n)}{d+2} {{n+d} \choose {d+1}} \\[1ex]
 =& \frac{(d+1)(1-n)(n+d)}{(d+2)(d+1)} {{n+d-1} \choose d}\\[1ex]
 =& \frac{(n+d)}{d+2} \frac{(1-n)(d+1)}{d+1}{{n+d-1} \choose d}\\[1ex]
 \leq &  \frac{d(1-n)}{d+1}{{n+d-1} \choose d} \leq -n^2.\\
\end{align*}
We have now shown the inequality
\[
 \left( \frac{n+d}{d+1}-n \right) \!\! {{n+d-1} \choose d} \leq -n^2.
\]
Hence the ideal $I$ satisfies the required properties in the theorem. This means that we have proved the following corollary.\\

\begin{coro}
 Let $R=\mathbb{C}[x_1, \ldots, x_n]$, and $I$ be an ideal in $R$ genereted by $k$ generic forms of degree $d$, where $d \geq 2$ and $n \geq 4$. Then Conjecture 1 is true, that is 
\[
  H_{R/I}(t)=\left\lceil \frac{(1-t^d)^k}{(1-t)^n} \right\rceil,
 \]
when ${{n+d-1} \choose d}-n < k \leq {{n+d-1} \choose d}$. \\
\end{coro}
  
One might ask for how small $k$ it is possible to consider ideals generated by monomials, as we did above. The following example shows that it is not always possible to use an ideal generated by monomials for small $k$. \\

\begin{exmp}
 Let $R=\mathbb{C}[x,y,z,w]$ and consider ideals generated by five forms of degree two. The conjectured Hilbert series is then
 \[
  \left\lceil \frac{(1-t^2)^5}{(1-t)^4} \right\rceil = 1+4t+5t^2
 \]
Let $\mathfrak{m}=(x,y,z,w)$. From the series above we see that $\mathfrak{m}^3 \subset I$. Note that $x^2$, $y^2$, $z^2$, and $w^2$ must be included in the ideal. If, for example $x^2 \not\in I$, then $x^3 \notin I$. Assume therefore $I=(x^2,y^2,z^2,w^2,m)$ for some monomial $m$ of degree two. Without loss of generality we can assume that $m=xy$. But then $xzw \notin I$, which contradicts $\mathfrak{m}^3 \subset I$.

\end{exmp}


In fact, one can prove that if $k>n$, and the Hilbert series has a positive coefficient for $t^{2d-1}$, then the ideal can not be generated by monomials. As we saw in the example, $x_1^d, \ldots, x_n^d \in I$. Assume $I=(x_1^d, \ldots, x_n^d, m_1, \ldots, m_{n-k})$, where $m_1, \ldots, m_{n-k}$ are monomials of degree $d$. If the Hilbertseries is the conjectured series, then the coefficient of $t^{2d-1}$ should be as small as possible. This means that we can not have a relation
\[
 f_1x_1^d + \dots + f_nx_n^d + f_{n+1}m_1 + \dots + f_km_{n-k}=0
\]
where $f_1, \ldots, f_k$ are forms of degree $d-1$. That is, all monomials that comes from multiplying one of the generators of $I$ by a monomial of degree $d-1$ have to be linearly independent. But this is, as we shall see, not true. Assume $m_1$ has a factor $x_i$. Then 
\[
 x_i^{d-1}m_1 = x_1^d \frac{m_1}{x_i},
\]
which proves that the monomials described above are not linearly independent. 

\section{Ideals generated by powers of generic forms}
As we saw above, it is not always possible to consider ideals generated by monomials. It is of course much more complicated to compute the Hilbert series given by an ideal genereted by generic forms. Even with the help of computers, the computations get to heavy when the number of variables, degree, or number of forms get too large. One way to make the computations slightly easier is to use powers of generic forms. That is, we consider ideals generated by $g_1^m, \ldots, g_k^m$, for some $m \in \mathbb{N}$, where the $g_i$'s are generic forms of degree $d$. There is reason to belive that ideals of this type gives the conjectured Hilbert series for ideals generated by generic forms of degree $md$. It is proved in \cite{Geramita-Schenck} that when $n=2$, powers of generic linear forms are generic. Hence we can choose our $g_i$'s to be $d$'th powers of linear forms. Then the $g_i^m$'s are again powers of linear forms, and it follows that these are also generic. However, it is well known that $d$'th powers of generic linear forms does not always give the same Hilbert series as generic forms of degree $d$ if $n\ge3$.\\

\begin{conj}\label{genpow}
 Assume that $g_1, g_2,\ldots, g_k$ are generic forms of degree $d \ge 2$ in $R=\mathbb{C}[x_1, \ldots, x_n]$. Let $I=(g_1^m,g_2^m, \ldots, g_k^m)$ be an ideal of $R$. Then 
 \[
  H_{R/I}(t)=\left\lceil \frac{(1-t^{md})^k}{(1-t)^n} \right\rceil.
 \]
\end{conj}

Another way to make the computations a little bit faster is to use ideals generated by $x_1^{md}, \ldots, x_n^{md}$, together with powers of random forms of the given degree. We believe that if $g_1, \ldots, g_k$ are generic forms of degree $d$, then $(g_1, \ldots, g_k)$ and $(x_1^d, \ldots, x_n^d, g_{n+1}, \ldots g_k)$ always give the same Hilbert series. 

As noted earlier, it is unneccesary to use more than ${{n+md-1} \choose md}$ forms. Given $n$, $m$, and $d$, there is only a finite number of ideals we need to try, in order to prove that the Hilbert series is the expected series. But even for small $n$, $m$, and $d$, this might be time consuming. Fortunately, there are ways to reduce the number of cases that need to be tested. This is illustrated in the following example. \\

\begin{exmp}
Let $R=\mathbb{C}[x,y,z]$, and say that we are now working with ideals generated by quadratic forms raised to the power of seven. Assume that we have found a set of 26 forms that gives rise to the expected Hilbert series, which is
\[
 \sum_{i=0}^{13} {{i+2} \choose i} t^i + 94t^{14} +58t^{15}.
\]
We see here that the ideal generated by these 26 forms contains all forms of degree 16. Then, of course, an ideal generated by more than 26 generic forms will also contain all forms of degree 16. \\
Assume also that we have found a 45 forms $g_1^7, \ldots, g_{45}^7$ that gives the expected Hilbert series, which is
\[
 \sum_{i=0}^{13} {{i+2} \choose i}t^i + 75t^{14} +t^{15}.
\]
There are $ {17 \choose 2}=136$ monomials of degree 15 in $R$. From the above series we see that the forms of degree 15 in the ideal span a vector space of dimension $136-1=135$ over $\mathbb C$. This space is spanned by $\{xg_j^7, yg_j^7, zg_j^7\}_{j=1}^{45}$. These are $3 \cdot 45=135$ forms, which means that they are linearly independent. Then $\{xg_j^7, yg_j^7, zg_j^7\}$ are also linearly independent when we use less than 45 generic forms. \\
This means that $I=(g_1^7, \ldots, g_k^7)$ always gives the smallest possible Hilbert series, when $26 \le k \le 45$. This series is neccesarily the conjectured series. 
\end{exmp}

We have calculated the Hilbert series for some choises of $n$, $m$, and $d$, using Macaulay2 \cite{M2}. All tested cases have given the expected series. For $n=3$ we tested all cases where $dm\leq 20$ and $d \geq 2$. Note that if we have proved for example the case $d=2$ and $m=8$, then the case $d=m=4$ follows. This is simply because a linear combination of $g_i^8$'s, with $\deg g_i=2$, can be considered as a linear combination of $(g_i^2)^4$'s. Hence there was no need to try all combinations of $d$ and $m$ that gives $dm\leq 20$. The verified cases for four and five variables are shown in the table below.  \\

\begin{center}
\begin{tabular}{l|lllllllllll}
\textbf{\textit{n}} & 4 & 4 & 4 & 4 & 4 & 5  \\\hline
\textbf{\textit{d}} & 2 & 2 & 3 & 2 & 3 & 2  \\\hline
\textbf{\textit{m}} & 2 & 3 & 2 & 4 & 3 & 2  
\end{tabular}
\end{center}

Note that does not only prove Conjecture \ref{genpow} in these cases, but does also give the following corollary.  \\

\begin{coro}
Let $R=\mathbb{C}[x_1, \ldots, x_n]$, and $I$ be an ideal in $R$ genereted by $k$ generic forms of degree $d$. Then Conjecture 1 is true, that is 
\[
  H_{R/I}(t)=\left\lceil \frac{(1-t^d)^k}{(1-t)^n} \right\rceil,
 \]
in the following five cases:
\begin{itemize}
 \item $n=4$, and $d=4,6,8,9$,
 \item $n=5$, and $d=4$.
\end{itemize}

\end{coro}


\begin{thebibliography}{9}
  
\bibitem{Anick}
 D. J. Anick
 \emph{Thin algebras of embedding dimension three}.
 Journal of Algebra 100: 235-259,
 1986.
 
\bibitem{Frob._ineq}
  R. Fröberg.
   \emph{An inequality for Hilbert series of graded algebras}.
   Mathematica Scandenavica 56: 117-144,
   1985.
   
 
\bibitem{Frob._Lofw.}
 R. Fröberg and C. Löfwall.
 \emph{On Hilbert series for commutative and noncommutative graded algebras}.
 Journal of Pure and Applied Algebra 76: 33-38,
 1991.
 
\bibitem{Geramita-Schenck}
 A. V. Geramita and H. Schenck
 \emph{Fat Points, Inverse Systems, and Piecewise Polynomial Functions}.
 Journal of Algebra 204: 116-128,
 1998.

 \bibitem{M2}
 Grayson, Daniel R. and Stillman, Michael E.
 \emph{Macaulay2, a software system for research
                   in algebraic geometry}.
 Available at http://www.math.uiuc.edu/Macaulay2/

 
 \end{thebibliography}
\end{document}